\documentclass[12pt,reqno]{article}

\usepackage[usenames]{color}
\usepackage{amssymb}
\usepackage{amsmath}
\usepackage{amsthm}
\usepackage{amsfonts}
\usepackage{amscd}
\usepackage{graphicx}

\usepackage[colorlinks=true,
linkcolor=webgreen,
filecolor=webbrown,
citecolor=webgreen]{hyperref}

\definecolor{webgreen}{rgb}{0,.5,0}
\definecolor{webbrown}{rgb}{.6,0,0}

\usepackage{color}
\usepackage{fullpage}
\usepackage{float}

\usepackage{graphics}
\usepackage{latexsym}
\usepackage{epsf}
\usepackage{breakurl}

\begin{document}
\theoremstyle{plain}
\newtheorem{theorem}{Theorem}
\newtheorem{corollary}[theorem]{Corollary}
\newtheorem{lemma}[theorem]{Lemma}
\newtheorem{proposition}[theorem]{Proposition}

\newtheorem{definition}[theorem]{Definition}
\newtheorem{example}[theorem]{Example}
\newtheorem{conjecture}[theorem]{Conjecture}

\theoremstyle{remark}
\newtheorem{remark}[theorem]{Remark}

\begin{center}
{\LARGE\bf Classification of Reflective Numbers}
\vskip 1cm
\large
Mahmoud Affouf\\
Department of Mathematics\\
Kean University\\
Union, NJ 07083 \\
USA\\
\href{mailto:maffouf@kean.edu}{\tt maffouf@kean.edu}
\end{center}

\vskip .2 in

\begin{abstract}
We define reflective numbers and their iterative summations. We provide classification of reflective numbers based on their iterative cyclical limits.
\end{abstract}

\section{Introduction}

Number relationships have boundless peculiarities. Some mathematicians have searched for these patterns while others have stumbled upon them. Some of these peculiarities are stated as conjectures and some of these conjectures can appeal even to amateurs.  There are quite a large number of easily stated conjectures and problems which contribute to increased interest and make progress in number theory. 
  Some conjectures have profound effects on mathematics, while others have fun and curious effects.  In 1949 D. R. Kaprekar proposed several ideas and schemes to recreational mathematics Gardner \cite{Gar} and Posamentier \cite{Pos}. He discovered a curious procedure between the number 6174 and 4-digit numbers. The Kaprekar number 6174 arises when one takes any four-digit number whose digits are not all identical. Rearrange the digits to form the largest and smallest four-digit numbers possible and then subtract these two new numbers and repeat this process. The result is always 6174 in at most seven steps. For example, starting with 5274:

\begin{align*}
7542-2457 = 5085; 8550-558 &= 7992;\\
9972-2799 = 7173; 7731-1377 &= 6354; \\
6543-3456 = 3087; 8730-378 &= 8352;\\
8532-2358 = 6174; 7641-1467&=6174
\end{align*}

There are various generalizations to the Kaprekar constant, see Devlin \cite{Dev}, Dolan \cite{Dol}, and Yamagami \cite{Yam}.

In this paper, we propose a scheme to generate integers. The scheme is based on the summation of the pairs of reflective numbers as defined in \cite{Ma}. Several interesting results related to a scheme "Reverse and Subtract" have been published by Greaney \cite{Mig}, Chandler \cite{Rac} and Sloane \cite{Sla}. Some of their  computations are similar to the results in this paper.  We will present our computational  numerical results, their visual descriptions, general observations and classifications. 

\section{ Definition and Reflective Scheme}

\begin{definition}
The \textbf{reflection} of any integer with n-digits is the integer obtained by reversing both its digits and sign, that is; the reflection of any given integer with n-digits $a=a_{n-1}a_{n-2}\cdots a_1a_0$  is the integer  $ r= -a_0a_1\cdots a_{n-2}a_{n-1}$. For example, if $a=328$, then its reflection is $r = -823$, and if $ a=-7250 $, then its reflection is $ r =  +0527=527$. We refer to a number and its reflective number as a \textbf{reflective pair}.

\end{definition}
{\bf Reflective Scheme}. We construct an iterative reflective sequence for any integer $a$ by the summation of its reflective pair using the explicit iterative formula:
$$ u_k = u_{k-1}+ r_{k-1}, \text{ for } k\geq1$$
where $u_0 = a$ and $r_0$ is the reflection of $u_0$. For example, the reflective sequence for $a= 571$ is computed as follows:

\begin{align}
u_0&=571, r_0=-175, u_1=u_0+r_0=571-175= 396\\
u_1&=396, r_1=-693, u_2=u_1+r_1=396-693=-297\\
u_2&=-297, r_2=792, u_3=u_2+r_2=-297+ 792=495\\
u_3&=495, r_3=-594, u_4=u_3+r_3=495-594=-99\\
u_4&=-99, r_4=99, u_5=u_4+r_4= -99+99=0
\end{align}
Thus, the reflective sequence for $571$ is $u=\{396, -297, 495, -99, 0\}$.

We observe that the numbers in a reflective sequence are multiples of 9. We will provide a proof for this observation.

\section{Preliminary Lemma}

\begin{lemma}
 The numbers in a reflective sequence of any integer with:
\begin{enumerate}
\item an even number of digits are divisible by 9.
\item an odd number of digits are divisible by 99.
\end{enumerate}

\end{lemma}

\begin{proof}

The first number in the reflective sequence of any integer number with n-digits $a=a_{n-1}a_{n-2}\cdots a_1a_0$ is
$u = a_{n-1}a_{n-2}\cdots a_1a_0-a_0a_1\cdots a_{n-2}a_{n-1}   $. First, we show the proof for $n=3$ and $n=4$. 

The odd case $n=3$: $u=a_2a_1a_0-a_0a_1a_2$ Express the numbers in powers of $10$.
$u=a_2\times 10^2 +a_1\times 10^1 +a_0-a_0\times 10^2 -a_1\times 10^1 - a_2$. Combine the coefficients as follows

\begin{align}
u=&a_2\times (10^2-1)+a_1\times (10^1-10^1)-a_0\times(10^2-1)\\
=&(a_2-a_0)\times (10^2-1)=(a_2-a_0)\times (99).
\end{align}
This number is a multiple of $99$.

The even case $n=4$: $u=a_3a_2a_1a_0-a_0a_1a_2a_3$ Express the numbers in powers of $10$.

$u=a_3\times 10^3 +a_2\times 10^2 +a_1\times 10^1 +a_0-a_0\times 10^3 -a_1\times 10^2 - a_2\times 10^1-a_3$. Combine the coefficients as follows

\begin{align}
u=&a_3\times (10^3-1)+a_2\times 10 \times (10^1-1)
\\
-&a_1\times 10 \times (10^1-1)-a_0\times(10^3-1)\\
=&a_3\times (999)+a_2\times 10 \times (9)\\
-&a_1\times 10 \times (9)-a_0\times(999).
\end{align}
This number is a multiple of $9$.

The proof follows the same steps for the general case. For any odd number $n=2k+1$, the reflective summation is  $u=a_{2k}a_{2k-1}\cdots a_1 a_0-a_0a_1\cdots a_{2k-1}a_{2k}$ Express the numbers in powers of $10$.
$u=a_{2k}\times 10^{2k}+a_{2k-1}\times 10^{2k-1}\cdots +a_1\times 10^1 +a_0-a_0\times 10^{2k} -a_1\times 10^{2k-1} - a_{2k-1}\times 10^1-a_{2k}$. We combine the coefficients as follows:
$u=a_{2k}(10^{2k}-1)+a_{2k-1}\times 10 (10^{2k-2}-1)+\cdots - a_1\times 10^1(10^{2k-2}-1)-a_0(10^{2k}-1) $.
The proof for the case of an even number of digits follows the same approach.
\end{proof}

\section{Classification of Reflective Sequences}

In this work, we focus on finding the limit of reflective sequences of integer. We use programming software to gain insight into the characteristics of iterative  reflective numbers. We will compute these sequences for any integer and compute the limits and the number of iteration to reach the limits.  

\textbf{1-digit numbers}: Clearly, the reflective sequence for each one-digit number is  zero. 

\textbf{Palindromic Numbers}: The reflective sequence for any palindromic number is zero.

\textbf{2-digit numbers}: There are 90 2-digit numbers 10 to 99. The direct computation shows that the limit of all reflective sequences is zero and it takes 1 to 6 iterations to reach the zero limit. The distribution of the frequency of numbers grouped by  the number of iteration is presented in Fig. 1. 

\textbf{3-digit numbers}: There are 900 3-digit numbers 100 to 999. The direct computation shows that the limit of all reflective sequences is zero and it takes 1 to 6 iterations to reach the zero limit. The distribution of the frequency of numbers grouped by  the number of iteration is presented in Fig. 1. 

\textbf{4-digit numbers}: There are 9000 4-digit numbers 1000 to 9999. The direct computation of the reflective iterations shows that there are two types of limits:

\begin{itemize}
 \item zero limit
 \item the cyclical limit of the pairs $\mp 2178, \mp 6534$ 
 \end{itemize} 
 
There are 8363 4-digit numbers whose reflective sequences converge to zero, and 537 numbers whose reflective sequences converge to the cyclical pairs: $\mp 2178$ and $\mp 6534$. We demonstrate these limits with two examples.
The reflective sequence of the number $7259$
is  $ r_i=\{-2268, 6354, 1818, -6363, -2727, 4545, -909, 0\}$:
\begin{align*}
7259 -9527&=-2268\\
-2268+8622&=6354\\
6354-4536&=1818\\
1818-8181&=-6363\\
-6363+3636&=-2727\\
-2727+7272&=4545\\
4545-5454&=-909\\
-909+909&=0
\end{align*}
The number of iteration to get to the limit is 8 iterations. Next, we compute the reflective sequence of $3817$:

\begin{align*}
3817-7183&=-3366\\
-3366+6633&=3267\\
3267-7623&=-4356\\
-4356+6534&=2178\\
2178-8712&=-6534\\
-6534+4356&=-2178\\
-2178+8712&=6534\\
6534-4356&=2178
\end{align*}

Note the cyclical limit: $ 2178, -6534, -2178, 6534, 2178$. For this sequence, we define the number of iteration to be 4 since there are four reflective numbers to reach the cyclical limit. These are that is $r_i=\{-3366,3267,-4356, 2178\}$.

The computational results presented in Figure 1, show that the number of reflective iterations for any 4-digit number  is:
\begin{enumerate}
 \item 1 to 4 iterations for cyclical limits
 \item 1 to 13 iterations for zero limit
 \end{enumerate} 

\begin{figure}[h!]
\begin{center}
\includegraphics[width=14cm, height=12cm,scale=0.5]{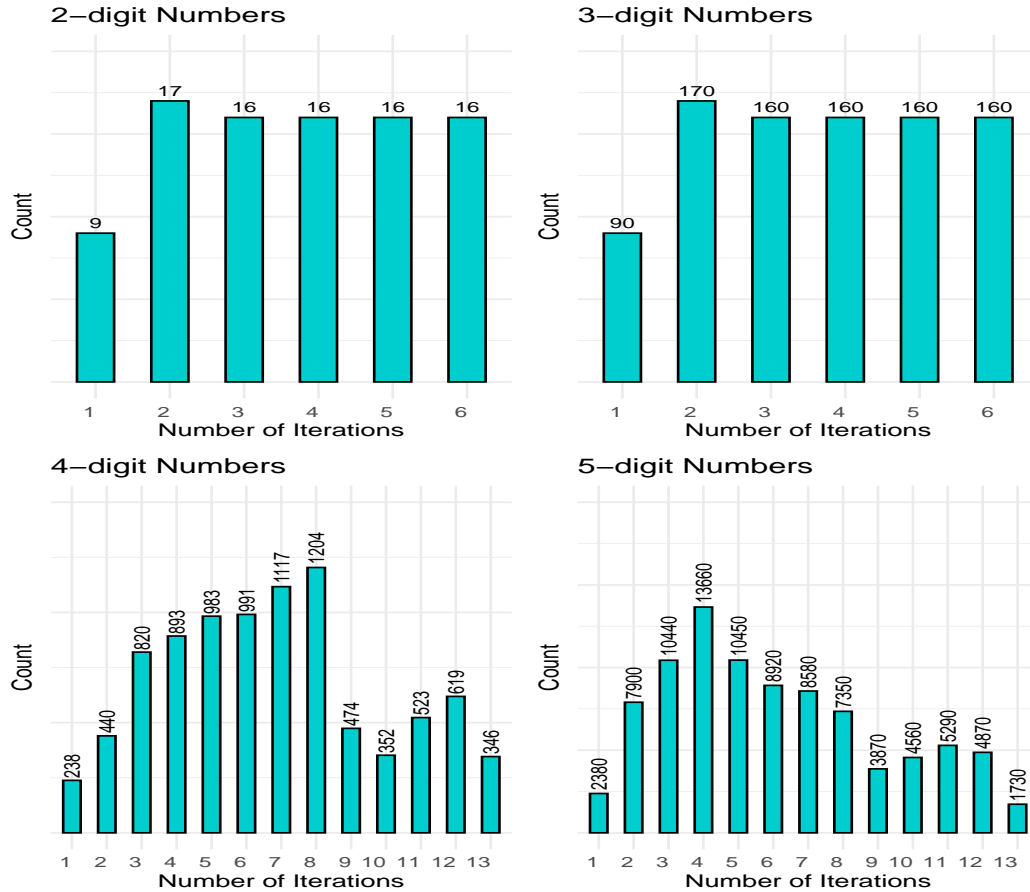}
\caption{The distributions of integers by number of iterations} \label{fig1}
\end{center}
\end{figure}

\textbf{5-digit numbers}: There are $90,000$ 5-digit numbers $10,000$ to $99,999$. The direct computation of the reflective iterations shows that there are three types of limits: 
\begin{itemize}
 \item zero limit
 \item the 4-digit cyclical limit of the pairs $\mp 2178, \mp 6534$ 
 \item the 5-digit cyclical limit of the pairs $\mp 21978, \mp 65934$ 
 \end{itemize}
 
The distribution of 5-digit numbers by their limits is presented in Figure 1. and table 1. The range of iterations to reach these limits is 1 to 13 iterations.

\textbf{6-digit numbers}: There are $900,000$ 6-digit numbers $100,000$ to $999,999$. The direct computation of the reflective iterations shows that there are four types of limits: 
\begin{itemize}
 \item zero limit
 \item the 4-digit cyclical limit of the pairs $\mp2178, \mp6534$ 
 \item the 5-digit cyclical limit of the pairs $\mp21978, \mp65934$ 
 \item the 6-digit cyclical limit of the pairs $\mp219978, \mp659934$ 
 \end{itemize}
 
The distribution of 6-digit numbers by their limits is presented in Figure 1 and table 1. The range of iterations to reach these limits is 1 to 49 iterations.
\begin{table}[h]
\centering

\begin{tabular}{|l|c|c|c|c|}
 \hline 
  & \textbf{zero} & \textbf{4-dig cycle} & \textbf{5-dig cycle}& \textbf{6-dig cycle}\\ 
 \hline
 \textbf{4-dig Numbers} & 8,363 (0.929\%) & 637 (0.071\%) & & \\
 \hline
 \textbf{5-dig Numbers} & 45,600(0.507\%) & 38,030 (0.423\%)& 6,370 (0.070\%)&  \\
 \hline
 \textbf{6-dig Numbers} & 460,458 (0.512\%) & 241,749 (0.269\%)& 178,686 (0.198\%)& 19,107 (0.021\%)\\
 \hline
\end{tabular}
\caption{Count of numbers by iterative limits}
\end{table}

\begin{figure}[h!]
\begin{center}
\includegraphics[width=16cm, height=12cm,scale=0.5]{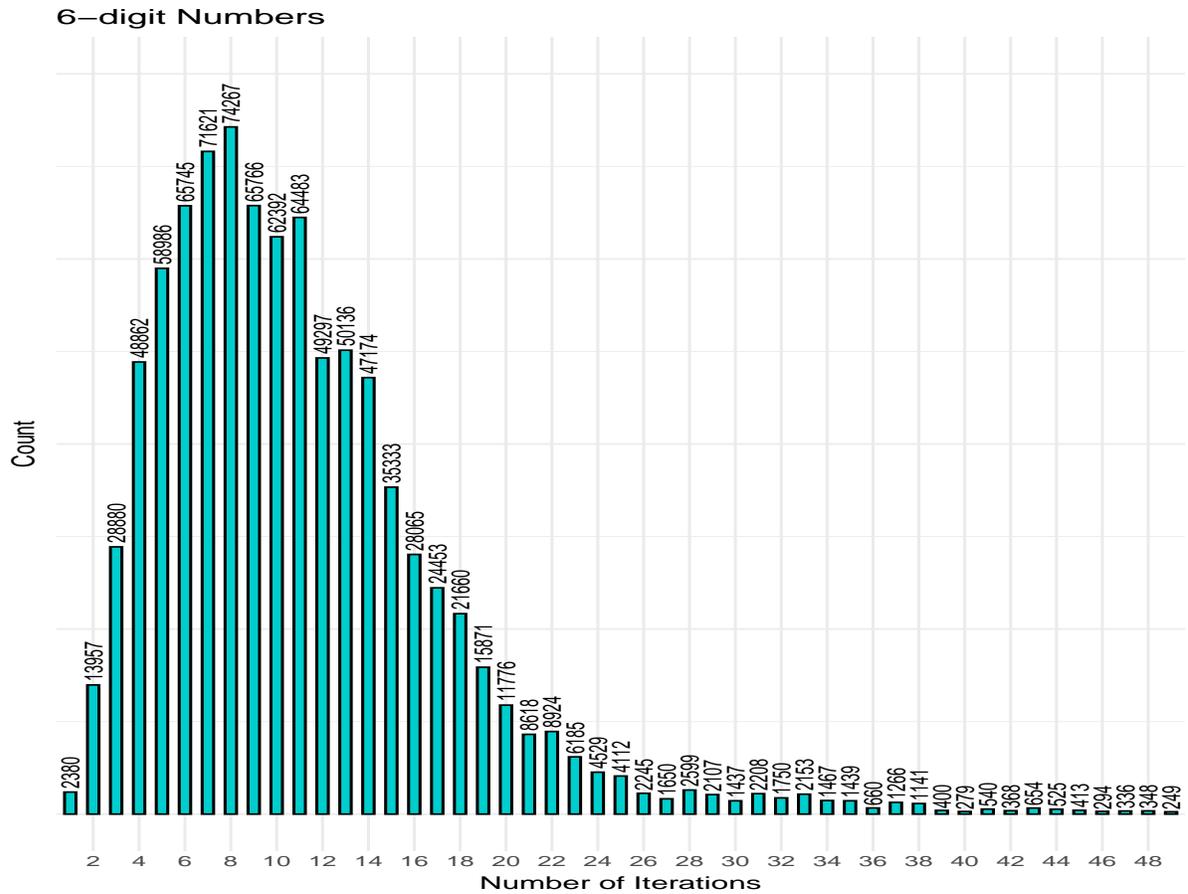}
\caption{The distribution of 6-digit integers  by number of iterations} \label{fig2}
\end{center}
\end{figure}

\textbf{7-digit numbers}: There are $9,000,000$ 7-digit numbers $1,000,000$ to $9,999,999$. A direct computation of the reflective iterations shows that there are five types of limits: 
\begin{itemize}
 \item zero limit
 \item the 4-digit, 5-digit, and 6-digit  cyclical limits of the pairs
 \\ $\mp2178, \mp6534,\\ \mp21978, \mp65934,\\ \mp219978, \mp659934$ 
  \item the 7-digit cyclical limit of the pairs $\mp2199978, \mp6599934$
 \end{itemize}
 
  The range of iterations to reach these limits is 1 to 49 iterations.

 \textbf{8-digit numbers}: There are $90,000,000$ 8-digit numbers $10,000,000$ to $99,999,999$. A direct computation of the reflective iterations shows that there are eight distinct cyclical limits: 
\begin{itemize}
 \item zero limit
 \item the 7-digit cyclical limits of the pairs 
 \\ $\mp2178, \mp6534,\\ \mp21978, \mp65934,\\ \mp219978, \mp659934 \\
 \mp2199978, \mp6599934$
\item the 8-digit cyclical limit of the pairs $\mp21999978, \mp65999934$ 
\item the 8-digit cyclical limit of the pairs $\mp21782178, \mp65346534$
\item the 8-digit cyclical limit of the number $11436678$ and its reflective iterations.

 \end{itemize}
 
  The range of iterations to reach these limits is 1 to 71 iterations. The last cyclical limit include the following sorted 14 numbers with their reflections:
  \begin{align*}
&11436678,13973058,19582398, 23981958, 30581397,&\\
&32662377, 33218856, 42464466, 44664246, 48737106,&\\
&61936974,
69746193,
71064873,
76226733&
\end{align*}

\textbf{9-digit numbers}:  A direct computation of the reflective iterations shows that there are eleven distinct cyclical limits: 
\begin{itemize}
	\item zero limit
	\item the 8-digit cyclical limits of the pairs 
	\\ $\mp2178, \mp6534,\\ \mp21978, \mp65934,\\ \mp219978, \mp659934 \\
	\mp2199978, \mp6599934\\
	\mp21999978, \mp65999934\\
	\mp21782178, \mp65346534\\
	11436678\quad \text{and its reflective iterations}\\
	$
	\item the 9-digit cyclical limit of the pairs $\mp219999978, \mp659999934$ 
	\item the 9-digit cyclical limit of the pairs $\mp217802178, \mp653406534$
	\item the 9-digit cyclical limit of the number $114396678$ and its reflective iterations.
	\end{itemize}
The cyclical limit of $114396678$ include the following sorted 14 numbers with their reflections:
\begin{align*}
	&114396678,139703058,195802398,239801958,305801397,&\\
	&326692377,332198856,424604466,446604246,487307106,&\\
	&619306974,697406193,710604873,762296733&
\end{align*}

\textbf{10-digit numbers}:  A direct computation of the reflective iterations shows that there are 14 distinct cyclical limits: 
\begin{itemize}
	\item zero limit
	\item the 8-digit cyclical limits of the pairs 
	\\ $\mp2178, \mp6534,\\ \mp21978, \mp65934,\\ \mp219978, \mp659934 \\
	\mp2199978, \mp6599934\\
	\mp21999978, \mp65999934\\
	\mp21782178, \mp65346534\\
	11436678\quad \text{and its reflective iterations}\\
	\mp219999978, \mp659999934\\
	\mp217802178, \mp653406534\\
	114369678\quad \text{and its reflective iterations}\\	$
	\item the 10-digit cyclical limit of the pairs $\mp2199999978, \mp6599999934$ 
	\item the 10-digit cyclical limit of the pairs $\mp2178002178, \mp6534006534$
	\item the 10-digit cyclical limit of the number $1143996678$ and its reflective iterations.
\end{itemize}
The cyclical limit of $1143996678$ include the following sorted 14 numbers with their reflections:
\begin{align*}
	&1143996678,1397903058,1958092398,2398091958,3058901397,&\\
	&3266992377,3321998856,4246004466,4466004246,4873007106,&\\
	&6193006974,6974006193,7106004873,7622996733&
\end{align*}
\section{Observations on the Cyclical Limits}

\begin{enumerate}
	\item The prime factorization of $2178, 21978, \cdots $, and their $198$ multiple patterns:
	\begin{align*}
		&2178=2\times 3^2\times 11^2&=198\times 11&=22\times(10^2-1)\\
		&21978=2\times 3^3\times 11\times 27&=198\times 111&=22\times(10^3-1)\\
		&219978=2\times 3^2\times 11^2\times 101&=198\times 1111&=22\times(10^4-1)\\
		&2199978=2\times 3^2\times 11\times 41\times 271&=198\times 11111&=22\times(10^5-1)\\
		&21999978=2\times 3^3\times 7\times 11^2\times 13\times 37&=198\times 111111&=22\times(10^6-1)\\
		&219999978=2\times 3^2\times 11\times 239\times 4649&=198\times 1111111&=22\times(10^7-1)\\
		&2199999978=2\times 3^2\times 11^2\times 73\times 101\times 137&=198\times 11111111&=22\times(10^8-1)\\
		\end{align*}
	These limits can be generated by the formula : $ 22\times (10^k-1)$ for $k\geq 2$.
	\item The prime factorization of $6534, 65934, \cdots$ , and their $594$ multiple patterns:
	\begin{align*}
		&6534=2\times 3^3\times 11^2&=594\times 11&=66\times(10^2-1)\\
		&65934=2\times 3^4\times 11\times 27&=594\times 111&=66\times(10^3-1)\\
		&659934=2\times 3^3\times 11^2\times 101&=594\times 1111&=66\times(10^4-1)\\
		&6599934=2\times 3^3\times 11\times 41\times 271&=594\times 11111&=66\times(10^5-1)\\
		&65999934=2\times 3^4\times 7\times 11^2\times 13\times 37&=594\times 111111&=66\times(10^6-1)\\
		&659999934=2\times 3^3\times 11\times 239\times 4649&=594\times 1111111&=66\times(10^7-1)\\
		&6599999934=2\times 3^3\times 11^2\times 73\times 101\times 137&=594\times 11111111&=66\times(10^8-1)\\
	\end{align*}
These limits can be generated by the formula : $ 66\times (10^k-1)$ for $k\geq 2$.
\item The prime factorization of $21782178, 217802178, \cdots $, and their $198$ multiple patterns:
\begin{align*}
	&21782178=2\times 3^2\times 11^2\times 73\times 137&=198\times 110011&=2178\times(10^4+1)\\
	&217802178=2\times 3^2\times 11^3\times 9091&=198\times 1100011&=2178\times(10^5+1)\\
	&2178002178=2\times 3^2\times 11^2\times 101\times 9901&=198\times 11000011&=2178\times(10^6+1)\\
	\end{align*}
These limits can be generated by the formula : $ 22\times (10^2-1) (10^k+1)$ for $k\geq 4$.
\item The prime factorization of $65346534, 653406534, \cdots$ , and their $594$ multiple patterns:
\begin{align*}
	&65346534=2\times 3^3\times 11^2\times 73\times 137&=594\times 110011&=6534\times(10^4+1)\\
	&653406534=2\times 3^3\times 11^3\times 9091&=594\times 1100011&=6534\times(10^5+1)\\
	&6534006534=2\times 3^3\times 11^2\times 101\times 9901&=594\times 11000011&=6534\times(10^6+1)\\
\end{align*}

These limits can be generated by the formula : $ 66\times (10^2-1) (10^k+1)$ for $k\geq 4$.

\item The prime factorization of $11436678, 114396678, \cdots$ , and their $198$ multiple patterns:
\begin{align*}
	&11436678=2\times 3^2\times 11^2\times 59\times 89&=198\times 57761\\
	&114396678=2\times 3^3\times 11\times 192587&=198\times 577761\\
	&1143996678=2\times 3^2\times 11^2\times 23\times41\times 557&=198\times 5777761\\
\end{align*}
The related $14$ cyclical reflective numbers have also multiple $198$  patterns.

\item the first number in each n-digit integers leading to cyclical pair limit is $ 10012, 100012, \cdots$  
\begin{align*}
	&10012=2^2\times 2503&=4\times 2503\\
	&100012=2^2\times 11\times 2273&=4\times 25003\\
	&1000012=2^2\times 13\times 19231&=4\times 250003\\
	\end{align*}
\item There are $637$ 4-digit numbers with cyclical limits. These numbers satisfy the explicit formula $$ n= 1012+ 11k$$
for $k\geq 1$. However, not every number generated by this formula has cyclical limit. For eaxample, the number $3278=1012+ 206\times 11 $ has zero limit. In addition, all palindromic numbers generated by this formula have zero limit. 

\end{enumerate}
\section{Curiosities}

The reflective limits have interesting patterns and amusing forms, here we list some of them:

\begin{enumerate}
\item $\dfrac{2178}{2+1+7+8}=121$
\item $\dfrac{6534}{6+5+3+4}=343=3\times 121$
\item $\dfrac{6534}{2178}=\dfrac{65934}{21978}=\dfrac{659934}{219978}=\dfrac{6599934}{2199978}=3$
\item $\dfrac{219978}{2178}=\dfrac{659934}{6534}=101$
\item $\dfrac{21999978}{2178}=\dfrac{65999934}{6534}=10101$
\item $\dfrac{21782178}{2178}=\dfrac{65346534}{6534}=10001$

\end{enumerate}

\section{Concluding Remarks}

We have introduced the definition of reflective numbers and their scheme, computationally isdentified the new cyclical limits and their generalized patterns up to
 10-digit numbers, and introduced few formulas for cyclical limits. There are many interesting open questions regarding the reflective numbers such as finding an algorithm to classify the integers based on their iterative reflective limits and identifying patterns of  these limits for all $n$-digit integers, and exploring these patterns for other base number systems. 

\section{Acknowledgments}
I thank J. Shallit for helpful comments, suggestions and 
relevant references that improved the quality of this work.

\end{document}